\documentclass[12pt,psamsfonts]{amsart}
\usepackage{amssymb,amsmath,amscd,eucal,epsfig,}
\usepackage{tikz}
\usetikzlibrary{arrows.meta}

\def\Dj{\hbox{D\kern-.73em\raise.30ex\hbox{-} \raise-.30ex\hbox{}}}
\def\dj{\hbox{d\kern-.33em\raise.80ex\hbox{-} \raise-.80ex\hbox{\kern-.40em}}}

\def\<{\langle}                     
\def\>{\rangle}                     

\newtheorem{thm}{Theorem}[section]

\newtheorem{cor}[thm]{Corollary}

\newcommand{\ben}{\begin{enumerate}}
	\newcommand{\een}{\end{enumerate}}

\theoremstyle{plain}
\newtheorem{theorem}{Theorem}[section]

\newtheorem{lemma}{Lemma}[section]

\newtheorem{remark}{Remark}[section]

\theoremstyle{definition}
\newtheorem{definition}{Definition}[section]

\numberwithin{equation}{section}
\begin{document}
	\title[Hamiltonian Complete Number of Graphs]{Hamiltonian Complete Number of Some Variants of Caterpillar Graphs} 
	
	
	\author[T.C Adefokun]{Tayo Charles Adefokun$^1$ }
	\address{$^1$Department of Computer and Mathematical Sciences,
		\newline \indent Crawford University,
		\newline \indent Nigeria}
	\email{tayoadefokun@crawforduniversity.edu.ng}
	
	\author[O. L. Ogundipe]{Opeoluwa Lawrence Ogundipe$^2$}
	\address{$^2$Department of Mathematics,
		\newline \indent University of Ibadan, Ibadan
		\newline \indent Nigeria}
	\email{opeogundipe2002@yahoo.com}
	
	\author[K.N. Onaiwu]{Kingsley Nosa Onaiwu$^3$}
	\address{$^3$Department of Physical and Earth Sciences,
		\newline \indent Crawford University,
		\newline \indent Nigeria}
		\email{littlebeging@yahoo.com}
	
\author[D. O. Ajayi]{Deborah Olayide Ajayi$^4$}
	\address{$^4$Department of Mathematics,
		\newline \indent University of Ibadan, Ibadan
		\newline \indent Nigeria}
	\email{adelaideajayi@yahoo.com}
	
	\keywords{Hamiltonian Graphs, Non-Hamiltonian Graphs, Caterpiller graphs. Spanning Paths Spanning cycles\\
		\indent 2010 {\it Mathematics Subject Classification}. Primary: 05C45, 05C40 05C38}
	
	\begin{abstract} 
		A graph \(G\) is said to be Hamiltonian if it contains a spanning cycle. In this work, we investigate the Hamiltonian completeness of certain classes of caterpillar graphs, which are trees with a central path to which all other vertices are adjacent. For a non-Hamiltonian graph \(G\), the Hamiltonian complete number \(\lambda_H(G)\) is the minimum number of edges that must be added to \(G\) to make it Hamiltonian. We focus on both regular and irregular caterpillar graphs, deriving explicit formulas for \(\lambda_H(G)\) in various cases. Specifically, we show that for a regular caterpillar graph \(G_{n(k)}\) where each vertex on the central path is adjacent to \(k\) leaves, \(\lambda_H(G_{n(k)}) = n(k-1)\). We also explore irregular caterpillar graphs, where the number of leaves adjacent to each vertex on the central path varies, and provide bounds for \(\lambda_H(G)\) in these cases. Our results contribute to the understanding of Hamiltonian properties in tree-like structures and have potential applications in network design and optimization.
	\end{abstract} 
	
	\maketitle		
	
	\section{Introduction}
	Let $G$ be a graph with vertex set $V(G)$ and edge set $E(G)$. A Hamiltonian cycle on a $G$ is a cycle that visits every vertex on $V(G)$ exactly once. Therefore, we say that $G$ is Hamiltonian if it contains a spanning cycle. Hamiltonian problem is the problem of determining if a graph contains a spanning cycle or not. This problem is $NP-$complete,  'notoriously hard' \cite{SB}, and the complete characterization of Hamiltonian graph has not been obtained. Some even believe that such characterization in non-existent. Furthermore, the only available technique of determining if a graph contains a spanning cycle is by exhaustive searching.
	
	Graphs such as complete graphs, cycle graphs, tournament on odd number vertices\cite{R1} are all Hamiltonian graphs. All Hamiltonian graphs are bi-connected, which means that at least two edges will be removed from the graph in other to dissolve it into at least two components. However, not every bi-connected graph is Hamiltonian. Attempts have been made on characterization of Hamiltonian graphs and the best known results are vertex degree characterization. Bondy-Chvatal in \cite{BC1} observed that a graph is Hamiltonian if and only if its closure is Hamiltonian. Therefore, any graph whose closure is complete is Hamiltonian. (The closure of  graph $G$ on $n$ vertices is obtained when every pair of non-adjacent vertices are connected by an edge not in $E(G)$. According to Dirac \cite{D1}, a simple graph $G$ on $n-$vertices is Hamiltonian if $d(v) \geq \frac{n}{2}$ for all $v \in V(G) $, and $n \geq 3$, while in \cite{O1} Ore showed that for any simple graph $G$ of order $n \geq 3$, $G$ is Hamiltonian if for all pair of non adjacent vertices, the sum of their degree is at least $n$. It was also observed on \cite{H1} that all $4-$connected planar triangulation contains a spanning cycle, while Tutte improved on this in \cite{T1} by showing that every $4-$connected planar graph is in fact Hamiltonian.   
	
	Graphs that do not contain a spanning cycle is known as non-Hamiltonian graph. Such, trivially, is any $1-$connected graph or any graph that contains a leaf or a pendant. Paths, trees etc are non-Hamiltonian graphs. In \cite{AA1} a result was obtained that showed some conditions for which a graph will not be Hamiltonian. Also, \cite{EN11} investigates a non-Hamiltonian $3-$connected cubic bipartite graph, where  non-Hamiltonian cyclical 4-edge connected bi-cubic graph was constructed on $54$ vertices. 
	
	In this work, we continue to explore the idea on non-Hamiltonian graphs as introduced in \cite{A1}, which was motivated by the importance of Hamiltonian graphs in computer graphics, mapping, electronic circuitry, route planning and etc. It was asked in \cite{A1} that how many edges can be added to the $E(G)$ of graph $G$ in order to make $G$ Hamiltonian? Trivially, for a Hamiltonian graph, the number is zero. The optimal number of such edges is called the Hamiltonian complete number $\lambda_H(G)$ of graph $G$. Hamiltonian complete number was investigated for some basic graphs in \cite{A1}. 
	
	In this work, we look further at some naturally non-Hamiltonian graphs with the aim of deriving their $\lambda_H(G)-$numbers. We focus on various caterpillar tree graphs.

	\section{Preliminaries}
	In this section we present some of the initial results and definitions that will be used in work. Other definition may be presented in the main body of the work as they are needed.
	
	A path $P_n$ and a cycle $C_n$ both contain $n$ vertices,  $n-1$ and $n$ edges respectively. The set $[n]$ denotes $\left\lbrace  1,2,...,n\right\rbrace $ while $[a,b]= \left\lbrace  a,a+1,...,b\right\rbrace$. We describe two vertices $a_1,a_2 \in V(G)$ as adjacent if $a_1a_2 \in E(G)$. The degree $d(v_i)$ of a vertex $v_i$ is the cardinal number of the set of vertices adjacent in $V(G)$ that are adjacent to  $v_i$. A vertex $v_i \in V(G)$ is called a leaf if $d(v_i)=1$ and the edge that connects leaf is known as a pendant. For some vertex $v_i \in V(G)$, $l(v_i)$ is the number of leaves adjacent to $v_i$. In other words, the number of pendants adjacent to $v_i$. 
	
	Also, the distance between two vertices, $a_1,a_2 \in V(G)$ is $d(a_1,a_2)=k$, which is the minimum number of edges between $a_1$ and $a_2$. 
	
	A star graph $S_n$ is a graph such that the central vertex $v_0$ is adjacent to $n$ leaves. It should be noted therefore that $|V(S_n)|=n+1$ and $|E(S_n)|=n$. We shall refer to Hamiltonian path complete number of $G$ as the number of edges need to be added to $E(G)$ for $G$ to contain a spanning path. We define Hamiltonian complete set $E^H(G)$ of graph $G$ as the optimal set of all extra edges $G$ required to be Hamiltonian. Clearly, the cardinal number $E^H(G)$ is the Hamiltonian complete number of $G$.
	
	The following results are established in \cite {A1}.
	
	\begin{theorem} \label{thm1}
		For a graph $G$ that $G$ contains $n$ leaves, $\lambda_H(G) \geq \lceil \frac{n}{2} \rceil$.
		
	\end{theorem}
	
	\begin{theorem} \label{thm2}
		Let $S_n$ be a star graph. Then $\lambda_H(S_n) = n-1$.
	\end{theorem}

	\section{Results}
	Now we present the results in this work. We start with the definition of a caterpillar tree graph.
	A caterpillar tree graph, $G$, is a connected graph that contains a central path $P_n$, to which all other vertices of $G$ are at distance at most $1$. We categorize caterpillar graphs that are being considered in this work into regular and irregular caterpillars. For regular caterpillar graphs, for every $v_i \in V(P_n)$, $l(v_i) = k$, where $k \geq 1$ and $k$ is an integer, where $P_n$ is the central path. In irregular caterpillar, for $v \in V(P_n)$, $l(v_i) \in [0,p]$, where $p$ is a positive integer.  We shall denote as $G_{n(k)}$, any caterpillar graph with central path $P_n$ whose each vertex on $P_n$ is adjacent to $k$ leaves. 
	
	In the first result, we consider a regular $G_{n(1)}$ with $l(v)=1$, for all $v \in V(P_n)$.
	
	\begin{theorem}
		Let $G_{n(1)}$ be a regular caterpillar graph such that for every $v_i \in P_n$, $l(v_i)=1.$ Then $\lambda_H(G_{n(1)}) = \lceil \frac{n}{2} \rceil$.
	\end{theorem}  
	
	\begin{proof}
		Suppose that $n$ is odd. Clearly, $G_{n(1)}$ contains $n-$leaves, From Theorem \ref{thm1}, $\lambda_H(G_{n(1)}) \geq \lceil \frac{n}{2} \rceil$ . Conversely, suppose that $V(P_n)=\left\lbrace v_1, v_2, \cdots , v_n \right\rbrace \subset V'(G_{n(1)})$. Then for each $v_i \in V(P_n)$, there exists $u_i \in V(G_{n(1)})$, $u_i \notin V(P_n)$, such that $v_iu_i$ is a pendant in $E(G_{n(1)})$ and $V(P_n) \cup V'(G_{n(1)})=V(G_n)$. Suppose further that there exist an edge set $E'(G)$, with $E'(G) \cap E(G_{n(1)}) = \emptyset$, such that $E'(G)$ contains edges that will make $ G_{n(1)} $ Hamiltonian. Suppose therefore, that the following edges $\left\lbrace u_1u_2, u_3u_4, \cdots u_{n-2}u_{n-1}, u_nv_1 \right\rbrace$ of $E'(G)$ are added to $ G_{n(1)} $. Then $|E'(G)|=\frac{n+1}{2}=\lceil \frac{n}{2} \rceil$, and now we have the cycle $v_1 \rightarrow u_1 \rightarrow u_2 \rightarrow v_2 \rightarrow v_3 \rightarrow u_3 \rightarrow \cdots \rightarrow v_{n-1} \rightarrow u_{n-1} \rightarrow v_n \rightarrow u_n \rightarrow v_n \rightarrow v_1$, which contains all members of $E'(G)$ and also is a spanning cycle of the resultant graph. Next we consider the situation where $n$ is even. Suppose that $P_n$ is the central path of $G_{n(1)}$ and $n$ is even. Following Theorem \ref{thm1}, $\lambda_H(G_n) \geq \lceil \frac{n}{2} \rceil$. Now, suppose we create an edge set $E^H(G_{n(1)})$ and populate $E^H(G_{n(1)})$ with Hamiltonian complete edges $u_2u_3, u_4u_5, u_6u_7, \cdots, u_{n-2}u_{n-1}, u_nu_1$ of $G_{n(1)}$. Clearly $|E^H(G_{n(1)})|= \frac{n}{2}$. Now, let $\bar{G}_{n(1)}$ be a graph, such that $E(\bar{G}_{n(1}))=E(G_{n(1)}) \cup E^H(G_n)$ and $V(G_{n(1)}) = V(\bar{G}_{n(1)})$. Then $\bar{G}_{n(1)}$ contains a spanning cycle, $C_n = u_1 \rightarrow v_1 \rightarrow v_2 \rightarrow u_2 \rightarrow \cdots v_{n-1} \rightarrow v_n \rightarrow u_n \rightarrow u_1$, which contains all members of $E^H(G_{n(1)})$. Thus, $\lambda_H(G_{n(1)}) \leq \lceil \frac{n}{2} \rceil$ for all positive integer $n$.
		
	\end{proof}

	\begin{figure}[h]
		\centering
		\begin{tikzpicture}[
			vertex/.style={circle, draw, fill=black, inner sep=2pt},
			every label/.style={font=\small}, 
			]
			\node[vertex, label=below:$v_1$] (v1) at (0,0) {};
			\node[vertex, label=below:$v_2$] (v2) at (2,0) {};
			\node[vertex, label=below:$v_3$] (v3) at (4,0) {};
			\node[vertex, label=below:$v_4$] (v4) at (6,0) {};
			\node[vertex, label=below:$v_5$] (v5) at (8,0) {};
			
			\node[vertex, label=above:$l_1$] (l1) at (0,1) {};
			\node[vertex, label=above:$l_2$] (l2) at (2,1) {};
			\node[vertex, label=above:$l_3$] (l3) at (4,1) {};
			\node[vertex, label=above:$l_4$] (l4) at (6,1) {};
			\node[vertex, label=above:$l_5$] (l5) at (8,1) {};
			
			\draw (v1) -- (v2) -- (v3) -- (v4) -- (v5);
			
			\draw (v1) -- (l1);
			\draw (v2) -- (l2);
			\draw (v3) -- (l3);
			\draw (v4) -- (l4);
			\draw (v5) -- (l5);
			
		\end{tikzpicture}
		\caption{Caterpillar Graph $G_5$}
		\label{fig1}
	\end{figure}
	
	\begin{figure}[h]
		\centering
		\begin{tikzpicture}[
			vertex/.style={circle, draw, fill=black, inner sep=2pt},
			every label/.style={font=\small}, 
			]
			\node[vertex, label=below:$v_1$] (v1) at (0,0) {};
			\node[vertex, label=below:$v_2$] (v2) at (2,0) {};
			\node[vertex, label=below:$v_3$] (v3) at (4,0) {};
			\node[vertex, label=below:$v_4$] (v4) at (6,0) {};
			\node[vertex, label=below:$v_5$] (v5) at (8,0) {};
			
			\node[vertex, label=above:$l_1$] (l1) at (0,1) {};
			\node[vertex, label=above:$l_2$] (l2) at (2,1) {};
			\node[vertex, label=above:$l_3$] (l3) at (4,1) {};
			\node[vertex, label=above:$l_4$] (l4) at (6,1) {};
			\node[vertex, label=above:$l_5$] (l5) at (8,1) {};
			
			\draw (v1) -- (v2) -- (v3) -- (v4) -- (v5);
			
			\draw [arrows = {-Latex[width'=0pt .5, length=10pt]}] (v1) -- (l1);
			\draw [arrows = {-Latex[width'=0pt .5, length=10pt]}] (l2) -- (v2);
			\draw [arrows = {-Latex[width'=0pt .5, length=10pt]}] (v2) -- (v3);
			\draw [arrows = {-Latex[width'=0pt .5, length=10pt]}] (v3) -- (l3);
			\draw [arrows = {-Latex[width'=0pt .5, length=10pt]}] (l4) -- (v4);
			\draw [arrows = {-Latex[width'=0pt .5, length=10pt]}] (v4) -- (v5);
			\draw [arrows = {-Latex[width'=0pt .5, length=10pt]}] (v5) -- (l5);
			
			\draw [arrows = {-Latex[width'=0pt .5, length=10pt]}] (l1) -- (l2);
			\draw [arrows = {-Latex[width'=0pt .5, length=10pt]}] (l3) -- (l4);
			\draw [arrows = {-Latex[width'=0pt .5, length=10pt]}] (l5) -- (v1);
		\end{tikzpicture}
		\caption{ $G_5$ in Figure \ref{fig1} has transformed into an Hamiltonian graph by three extra edges, showing that $\lambda_H(G_5) = 3.$}
		\label{fig2}
	\end{figure}

	\begin{figure}[h]
		\centering
		\begin{tikzpicture}[
			vertex/.style={circle, draw, fill=black, inner sep=2pt},
			every label/.style={font=\small}, 
			]
			\node[vertex, label=below:$v_1$] (v1) at (0,0) {};
			\node[vertex, label=below:$v_2$] (v2) at (2,0) {};
			\node[vertex, label=below:$v_3$] (v3) at (4,0) {};
			\node[vertex, label=below:$v_4$] (v4) at (6,0) {};
			
			\node[vertex, label=above:$l_1$] (l1) at (0,1) {};
			\node[vertex, label=above:$l_2$] (l2) at (2,1) {};
			\node[vertex, label=above:$l_3$] (l3) at (4,1) {};
			\node[vertex, label=above:$l_4$] (l4) at (6,1) {};
			
			\draw (v1) -- (v2) -- (v3) -- (v4);

			\draw (v1) -- (l1);
			\draw (v2) -- (l2);
			\draw (v3) -- (l3);
			\draw (v4) -- (l4);
			
		\end{tikzpicture}
		\caption{Caterpillar Graph $G_4$}
		\label{fig3}
	\end{figure}

	\begin{figure}[h]
		\centering]
		\begin{tikzpicture}[
			vertex/.style={circle, draw, fill=black, inner sep=2pt},
			every label/.style={font=\small}, 
			]
			\node[vertex, label=below:$v_1$] (v1) at (0,0) {};
			\node[vertex, label=below:$v_2$] (v2) at (2,0) {};
			\node[vertex, label=below:$v_3$] (v3) at (4,0) {};
			\node[vertex, label=below:$v_4$] (v4) at (6,0) {};
			
			\node[vertex, label=above:$l_1$] (l1) at (0,1.8) {};
			\node[vertex, label=above:$l_2$] (l2) at (2,1) {};
			\node[vertex, label=above:$l_3$] (l3) at (4,1) {};
			\node[vertex, label=above:$l_4$] (l4) at (6,1.8) {};
			
			\draw (v1) -- (v2) -- (v3) -- (v4);

			\draw (v2) -- (l2);
			\draw (v3) -- (l3);
			\draw (v4) -- (l4);
			\draw [arrows = {-Latex[width'=0pt .5, length=10pt]}] (v1) -- (l1);
			\draw [arrows = {-Latex[width'=0pt .5, length=10pt]}] (l2) -- (v2);
			\draw [arrows = {-Latex[width'=0pt .5, length=10pt]}] (v2) -- (v1);
			\draw [arrows = {-Latex[width'=0pt .5, length=10pt]}] (v3) -- (l3);
			\draw [arrows = {-Latex[width'=0pt .5, length=10pt]}] (l4) -- (v4);
			\draw [arrows = {-Latex[width'=0pt .5, length=10pt]}] (v4) -- (v3);
			
			\draw [arrows = {-Latex[width'=0pt .5, length=10pt]}] (l3) -- (l2);
			\draw [arrows = {-Latex[width'=0pt .5, length=10pt]}] (l1) -- (l4);
		\end{tikzpicture}
		\caption{ $G_4$ in Figure \ref{fig3} has transformed into an Hamiltonian graph by two extra edges, showing that $\lambda_H(G_4) = 2.$}
	\end{figure}

	Next, we consider another class of regular Caterpillar graph $G_{n(2)}$, for which of every vertex $v_i \in V(P_n)$, $l(v_i)=2$. Clearly, $|V(G_{n(2)})|=3n$. We examine the Hamiltonian complete number of $G_{n(2)}$. 
	
	\begin{theorem}
		Let $G_{n(2)}$ be a caterpillar graph such that for every $v_i \in P_n$, $l(v_i)=2$. Then, $\lambda_H( G_{n(2)} )=n$.
	\end{theorem} 
	\begin{proof}
		Let $V(P_n) = {v_1,v_2, \cdots, v_n}$ be the vertex set of the central path $P_n$ of $G_n$ and let $V''(G_{n(2)} )=\left\lbrace u_1,u_2, u_3, u_4, \cdots, u_{2n-1}, u_{2n} \right\rbrace $ be leaves on $G_{n(2)} $. Now, $|V''(G_{n(2)} )|=2n$ and by an earlier result, $\lambda_H (G_{n(2)} ) \leq \frac{2n}{2}=n$. Conversely, suppose that $E^H(G_{n(2)} )$ is the Hamiltonian complete edge set and suppose that $E^H(G_{n(2)} )=\left\lbrace u_iu_{i+1}: 1 < i < 2n \right\rbrace$ $\cup \left\lbrace u_1u_n\right\rbrace $  and $|E^H(G_{n(2)})|=n-1+1=n$. Now, with $E(G_{n(2)} ) \cup E^H(G_{n(2)} )$, there exists a spanning cycle $C_{3n}=u_1 \rightarrow v_1 \rightarrow u_2 \rightarrow u_3 \rightarrow v_2 \rightarrow u_4 \rightarrow u_5 \rightarrow v_3 \rightarrow \cdots u_{n-2} \rightarrow u_{n-1} \rightarrow v_n \rightarrow u_n \rightarrow u_1$, which implies that $G_{n(2)} $ is Hamiltonian and $\lambda_H (G_{n(2)} ) \leq n$ and therefore, $\lambda_H(G_{n(2)} ) = n$.
	\end{proof}

	\begin{figure}[h]
		\centering
		\begin{tikzpicture}[
			vertex/.style={circle, draw, fill=black, inner sep=2pt},
			every label/.style={font=\small}, 
			]
			\node[vertex, label=below:$v_1$] (v1) at (0,0) {};
			\node[vertex, label=below:$v_2$] (v2) at (2,0) {};
			\node[vertex, label=below:$v_3$] (v3) at (4,0) {};
			\node[vertex, label=below:$v_4$] (v4) at (6,0) {};
			\node[vertex, label=below:$v_5$] (v5) at (8,0) {};
			
			\node[vertex, label=above left:] (l1a) at (-0.35,0.7) {}; 
			\node[vertex, label=above right:] (l1b) at (0.35,0.7) {}; 
			
			\node[vertex, label=above left:] (l2a) at (1.65,0.7) {}; 
			\node[vertex, label=above right:] (l2b) at (2.35,0.7) {}; 
			
			\node[vertex, label=above left:] (l3a) at (3.65,0.7) {}; 
			\node[vertex, label=above right:] (l3b) at (4.35,0.7) {}; 
			
			\node[vertex, label=above left:] (l4a) at (5.65,0.7) {}; 
			\node[vertex, label=above right:] (l4b) at (6.35,0.7) {}; 
			
			\node[vertex, label=above left:] (l5a) at (7.65,0.7) {}; 
			\node[vertex, label=above right:] (l5b) at (8.35,0.7) {}; 
			
			\draw (v1) -- (v2) -- (v3) -- (v4) -- (v5);
			
			\draw (v1) -- (l1a);
			\draw (v1) -- (l1b);
			
			\draw (v2) -- (l2a);
			\draw (v2) -- (l2b);
			
			\draw (v3) -- (l3a);
			\draw (v3) -- (l3b);
			
			\draw (v4) -- (l4a);
			\draw (v4) -- (l4b);
			
			\draw (v5) -- (l5a);
			\draw (v5) -- (l5b);
		\end{tikzpicture}
		\caption{A Caterpillar Graph $G_{5(2)}$ with central path of 5 vertices, where each vertex is connected to two leaves.}
		\label{fig:caterpillar_graph}
	\end{figure}
	
	\begin{figure}[h]
		\centering
		\begin{tikzpicture}[
			vertex/.style={circle, draw, fill=black, inner sep=2pt},
			every label/.style={font=\small}, 
			]
			\node[vertex, label=below:$v_1$] (v1) at (0,0) {};
			\node[vertex, label=below:$v_2$] (v2) at (2,0) {};
			\node[vertex, label=below:$v_3$] (v3) at (4,0) {};
			\node[vertex, label=below:$v_4$] (v4) at (6,0) {};
			\node[vertex, label=below:$v_5$] (v5) at (8,0) {};
			
			\node[vertex, label=above left:] (l1a) at (-0.45,1.3) {}; 
			\node[vertex, label=above right:] (l1b) at (0.35,0.7) {};
			
			\node[vertex, label=above left:] (l2a) at (1.65,0.7) {};
			\node[vertex, label=above right:] (l2b) at (2.35,0.7) {}; 
			
			\node[vertex, label=above left:] (l3a) at (3.65,0.7) {};
			\node[vertex, label=above right:] (l3b) at (4.35,0.7) {}; 
			
			\node[vertex, label=above left:] (l4a) at (5.65,0.7) {};
			\node[vertex, label=above right:] (l4b) at (6.35,0.7) {}; 
			
			\node[vertex, label=above left:] (l5a) at (7.65,0.7) {};
			\node[vertex, label=above right:] (l5b) at (8.65,1.3) {};
			
			\draw (v1) -- (v2) -- (v3) -- (v4) -- (v5);
			
			\draw (v1) -- (l1a);
			\draw (v1) -- (l1b);
			
			\draw (v2) -- (l2a);
			\draw (v2) -- (l2b);
			
			\draw (v3) -- (l3a);
			\draw (v3) -- (l3b);
			
			\draw (v4) -- (l4a);
			\draw (v4) -- (l4b);
			
			\draw (v5) -- (l5a);
			\draw (v5) -- (l5b);
			\draw [arrows = {-Latex[width'=0pt .5, length=10pt]}] (l1a) -- (l5b);
			\draw [arrows = {-Latex[width'=0pt .5, length=10pt]}] (l2a) -- (l1b);
			\draw [arrows = {-Latex[width'=0pt .5, length=10pt]}] (l3a) -- (l2b);
			\draw [arrows = {-Latex[width'=0pt .5, length=10pt]}] (l4a) -- (l3b);
			\draw [arrows = {-Latex[width'=0pt .5, length=10pt]}] (l5a) -- (l4b);
			\draw [arrows = {-Latex[width'=0pt .5, length=10pt]}] (v5) -- (l5a);
			\draw [arrows = {-Latex[width'=0pt .5, length=10pt]}] (l4b) -- (v4);
			\draw [arrows = {-Latex[width'=0pt .5, length=10pt]}] (v4) -- (l4a);
			\draw [arrows = {-Latex[width'=0pt .5, length=10pt]}] (l3b) -- (v3);
			\draw [arrows = {-Latex[width'=0pt .5, length=10pt]}] (v3) -- (l3a);
			\draw [arrows = {-Latex[width'=0pt .5, length=10pt]}] (l2b) -- (v2);
			\draw [arrows = {-Latex[width'=0pt .5, length=10pt]}] (v2) -- (l2a);
			\draw [arrows = {-Latex[width'=0pt .5, length=10pt]}] (l1b) -- (v1);
			\draw [arrows = {-Latex[width'=0pt .5, length=10pt]}] (v1) -- (l1a);
		\end{tikzpicture}
		\caption{Graph $G_{5(2)}$ is transformed into an Hamiltonian graph with 5 extra edges, showing that $\lambda_H(G_{5(2)}) = 5$.}
		\label{fig6}
	\end{figure}

	Next, we consider the regular caterpillar tree $ G_{n(k)} $ on a central path $P_n$, such that for each $v_i \in P_n$, $l(v_i) \geq 3$.
	\begin{remark}
		Let $G_n$ be a non Hamiltonian graph and $\lambda_H(G_{n(k)})$ be the Hamiltonian complete number of $G_{n(k)}$. It is trivial to show that if $\delta_H(G_{n(k)})$ is the least number of external edges required for $G_n$ to contain a spanning path, then $\delta_H(G_{n(k)}) = \lambda_H (G_{n(k)})-1$.  
	\end{remark}
	
	\begin{definition}
		Let $G_{n(k)}$ be a caterpillar graph. Then $P(G_{n(k)}) \in E(G_{n(k)})$ denotes the set of pendants on $G_{n(k)}$ 
	\end{definition}
	
	\begin{definition}
		Let $v_i \in P_n$ on a caterpillar graph $G_{n(k)}$, such that $v_i$ is incident to $k \geq 2$ number of leaves. Then, $v_i$ and the adjacent leaves induce a claw $CL_{v_i}$ in $G_{n(k)}$.
	\end{definition}
	
	\begin{remark}
		It should be noted that a claw is a form of a star. So, in some parts of the work, a claw is interchanged with a star. 
	\end{remark}
	
	\begin{theorem}
		Let $ G_{n(k)} $ be a $k$-regular caterpillar with central path $P_n$ such that for each $v_i \in V(P_n), l(v_i) \geq 3$. Then, $\lambda_H(G_{n(k)})=n(k-1)$. 
	\end{theorem} 
	\begin{proof}
		Let $V(G_{n(k)})=V(G'_{n(k)}) \cup V(P_n)$, where $V(P_n)$ in the vertex set of the central path $P_n$ and $V(G'_{n(k)})$ the set of leaves such that $v_iu \in P(G_{n(k)})$, where $P(G_{n(k)})$ is the set of pendants on $G_{n(k)}$. Now let $V(P_n)=\left\lbrace v_1,v_2, \cdots , v_n \right\rbrace $ and $V(G'_n) = \left\lbrace u_1, u_2,u_3 \cdots u_{kn}\right\rbrace $ , where $k$ is the number of leaves incident to $v_i$ for all $i \in [1,n]$. Clearly, claw $CL_{v_i}$ is a star  and therefore by Theorem \ref{thm2} $\lambda_H(CL_{v_i})=k-1$, since $|V(CL_{v_i})|=k+1$. Thus, $\delta_H(CL_{v_i})=k-2$. Now, $\sum_{i=1}^{n}\delta_n(CL_{v_i})=n(k-2)$. It should be observed that $v_i$ is not an initial vertex or the terminal vertex of the spanning path in $CL_{v_i}$or else, $\delta_H(CL_{v_i})=k-1$. Thus, suppose we seek $\delta_n(G_{n(k)})$. Then, a class of $n-1$ edges exists in $E(G_{n(k)})$ such that each member connects adjacent claws in $G'_n(k)$. Furthermore, $v_iv_{i+1}$ is not a member of this class of edges in $E'(G_{n(k)})$. Therefore, $\delta_n(G_{n(k)}) \geq n(k+2)+n-1$. Thus, $\lambda_H(G_{n(k)}) \geq n(k-1)$. 
		\\Conversely, for $G_{n(k)}$, a $k-$regular caterpillar, let $V(P_n)=\left\lbrace v_1,v_2, \cdots, v_n\right\rbrace $ and the set of leaves adjacent to $v_i \in V(P_n)$ is $V(P(v_i))=\left\lbrace u_i(j):i\in[1,n], j \in [i,k]\right\rbrace $, such that $V(P(v_1))=\left\lbrace u _1(1), u_1(2), \cdots, u_1(k) \right\rbrace $. Let an Hamiltonian complete edge set $E^H(G_{n(k)})$, such that $E^H(G_{n(k)}) \cap E(G_{n(k)})= \emptyset$. Now, we trace a spanning path $P^i_{k+1}$ on each claw $CL_{v_i}$ without loss of generality such that $P^i_{k+i}=u_i(1) \rightarrow v_i \rightarrow u_i(2) \rightarrow u_i(3) \rightarrow, \cdots, u_i(k)$. It can be seen that $u_i(2)u_i(3), u_i(3)u_i(4),...,u_i(k-1)u_i(k)$ are $k-2$ in number in $E^H(G_{n(k)})$. For all the $n-$claws on $G_{n(k)}$, therefore there are at least  $n(k-2)$ edges in $E^H(G_{n(k)})$. Let there exists $u_i(k)u_{i+1}(1)$, for all $i \in [1,n-1]$. The subset $u_i(k)u_{i+1}(1)$ of $E^H(G_{n(k)})$contains $n-1$ edges. Thus there exists a spanning path $P_{(k+1)-1}$ of $Gn(k)$, such that $P_{(k+1)-1}= u_1(1) \rightarrow v_1 \rightarrow u_1(2) \rightarrow u_1(3) \rightarrow, \cdots u_1(k) \rightarrow u_2(1) \rightarrow v_2 \rightarrow u_2(2) \rightarrow u_2(3) \rightarrow, \cdots, u_{n-1}(k) \rightarrow u_n(1) \rightarrow v_n \rightarrow u_n(2) \rightarrow u_n(3) \rightarrow, \cdots , u_n(k) $. Thus, the resultant graph contains a spanning path. Now, suppose $E^H(G_{n(k)})$ contains an edge $u_n(k)u_1(1)$, then the cardinal number of $E^H(G_{n(k)})=n(k-1)$. Now, let $\bar{G}_{n(k)}=G_{n(k)} \cup E^H(G_{n(k)})$. Then $\bar{G_{n(k)}}$ contains a spanning cycle on $V(G_{n(k)})$, and thus, $\lambda_H(G_{n(k)}) \leq n(k-1)$. 
		
	\end{proof}
	Now we have completed the investigation into finding the Hamiltonian complete number for all regular caterpillar graphs with central path $P_n$ for which each vertex of $P_n$ is adjacent to $k$ leaves, where $k \geq 1$. Next we define irregular caterpillar graphs and investigate the Hamiltonian complete number of these class of graphs.  
	
	\section{Hamiltonian Complete number of Irregular caterpillar graphs}
	We define an irregular caterpillar, $G^i_n$ as a caterpillar on a central path $P_n$ such that for $v_i \in V(P_n)$, $l(v_i)=k$, where $k \in [0,r]$, where $r$ is a positive integer. We shall consider various cases.
	\begin{remark} \normalfont
		Let us consider two caterpillar graphs $ G^i_{3(a)} $ and $G^i_{3(b)}$ both on central path $P_3$, such that for $G^i_{3(a)}$, $l(v_1) = l(v_3) = 1$ and $l(v_2) = 2$, while for $G^i_{3(b)}$, $l(v_1) = l(v_2) = 1$ and $l(v_3) = 2$. (Or $l(v_1) = 2$ and $l(v_2) = l(v_3) =1$). It is easy to confirm that even though $G^i_{3(a)}$ and $G^i_{3(b)}$  contain the same number of pendants, $\lambda_H(G^i_{3(a)})=3$ and $\lambda_H(G^i_{3(a)})=2$. Hence caterpillar graphs that contain subgraph $G^i_{3(a)}$ may have different Hamiltonian complete number from those that do not contain $G^i_{3(a)}$, even if the graphs contain similar central path and the same number of pendants. 		
	\end{remark}

	\begin{figure}[ht]
		\centering
		
		\begin{tikzpicture}[vertex/.style={circle, draw, fill=black, inner sep=1.5pt}]
			\node[vertex] (v1) at (0,0) {};
			\node[vertex] (v2) at (2,0) {}; 
			\node[vertex] (v3) at (4,0) {}; 
			
			\draw (v1) -- (v2) -- (v3);
			
			\node[vertex] (l1) at (0,1) {}; 
			\node[vertex] (l2a) at (1.5,1) {}; 
			\node[vertex] (l2b) at (2.5,1) {}; 
			\node[vertex] (l3) at (4,1) {}; 
			
			\draw (v1) -- (l1);
			\draw (v2) -- (l2a);
			\draw (v2) -- (l2b); 
			\draw (v3) -- (l3);
			
			\node[below] at (v1) {$v_1$};
			\node[below] at (v2) {$v_2$};
			\node[below] at (v3) {$v_3$};
		\end{tikzpicture}
		\hspace{2cm} 
		\begin{tikzpicture}[vertex/.style={circle, draw, fill=black, inner sep=1.5pt}]
			\node[vertex] (u1) at (0,0) {};
			\node[vertex] (u2) at (2,0) {}; 
			\node[vertex] (u3) at (4,0) {}; 
			
			\draw (u1) -- (u2) -- (u3);
			
			\node[vertex] (m1) at (0,1) {}; 
			\node[vertex] (m2) at (2,1) {}; 
			\node[vertex] (m3a) at (3.5,1) {}; 
			\node[vertex] (m3b) at (4.5,1) {}; 
			
			\draw (u1) -- (m1);
			\draw (u2) -- (m2);
			\draw (u3) -- (m3a);
			\draw (u3) -- (m3b); 
			
			\node[below] at (u1) {$u_1$};
			\node[below] at (u2) {$u_2$};
			\node[below] at (u3) {$u_3$};
		\end{tikzpicture}
		
		\caption{Caterpillar graphs $G^i_{3(a)}$ and $G^i_{3(b)}$ have the same number of leaves but different Hamiltonian complete numbers as shown in the Figure \ref{fig8}.}
		\label{fig7}
	\end{figure}

	\begin{figure}[ht]
		\centering
		
		\begin{tikzpicture}[vertex/.style={circle, draw, fill=black, inner sep=1.5pt}]
			\node[vertex] (v1) at (0,0) {};
			\node[vertex] (v2) at (2,0.5) {}; 
			\node[vertex] (v3) at (4,0) {}; 
			
			\draw (v1) -- (v2) -- (v3);
			
			\node[vertex] (l1) at (0,1) {}; 
			\node[vertex] (l2a) at (1.5,1) {}; 
			\node[vertex] (l2b) at (2.5,1) {}; 
			\node[vertex] (l3) at (4,1) {}; 
			
			\draw (v1) -- (l1);
			\draw (v2) -- (l2a);
			\draw (v2) -- (l2b); 
			\draw (v3) -- (l3);
			\draw [arrows = {-Latex[width'=0pt .5, length=10pt]}] (l1) -- (l2a);
			\draw [arrows = {-Latex[width'=0pt .5, length=10pt]}] (l2b) -- (l3);
			\draw [arrows = {-Latex[width'=0pt .5, length=10pt]}] (v3) -- (v1);
			\draw [arrows = {-Latex[width'=0pt .5, length=10pt]}] (v1) -- (l1);
			\draw [arrows = {-Latex[width'=0pt .5, length=10pt]}] (l2a) -- (v2);
			\draw [arrows = {-Latex[width'=0pt .5, length=10pt]}] (v2) -- (l2b);
			\draw [arrows = {-Latex[width'=0pt .5, length=10pt]}] (l3) -- (v3);
			
			\node[below] at (v1) {$v_1$};
			\node[below] at (v2) {$v_2$};
			\node[below] at (v3) {$v_3$};
		\end{tikzpicture}
		\hspace{2cm} 
		\begin{tikzpicture}[vertex/.style={circle, draw, fill=black, inner sep=1.5pt}]
			\node[vertex] (u1) at (0,0) {};
			\node[vertex] (u2) at (2,0) {}; 
			\node[vertex] (u3) at (4,0) {}; 
			
			\draw (u1) -- (u2) -- (u3);
			
			\node[vertex] (m1) at (0,2) {}; 
			\node[vertex] (m2) at (2,2) {}; 
			\node[vertex] (m3a) at (3.5,1) {}; 
			\node[vertex] (m3b) at (4.5,1) {}; 
			
			\draw (u1) -- (m1);
			\draw (u2) -- (m2);
			\draw (u3) -- (m3a);
			\draw (u3) -- (m3b); 
			\draw [arrows = {-Latex[width'=0pt .5, length=10pt]}] (m1) -- (m3a);
			\draw [arrows = {-Latex[width'=0pt .5, length=10pt]}] (m3b) -- (m2);
			\draw [arrows = {-Latex[width'=0pt .5, length=10pt]}] (u1) -- (m1);
			\draw [arrows = {-Latex[width'=0pt .5, length=10pt]}] (m3a) -- (u3);
			\draw [arrows = {-Latex[width'=0pt .5, length=10pt]}] (u3) -- (m3b);
			\draw [arrows = {-Latex[width'=0pt .5, length=10pt]}] (m2) -- (u2);
			\draw [arrows = {-Latex[width'=0pt .5, length=10pt]}] (u2) -- (u1);
			\node[below] at (u1) {$u_1$};
			\node[below] at (u2) {$u_2$};
			\node[below] at (u3) {$u_3$};
		\end{tikzpicture}
		
		\caption{The above shows $\lambda_H(G^i_{3(a)}) = 3$ and $\lambda_H(G^i_{3(b)}) =2$ }
		\label{fig8}
	\end{figure}

	\begin{remark} \normalfont
		We define $U$ as the set of all subgraphs $G^i_{3(a)}$ in caterpillar graph $G^i_n$ where for every $v_i \in P_n$, the central path of $G^i_n$, $l(v_i) \in [1,2]$. We set $|U|$ as the non-negative  cardinal number of $U$.
	\end{remark}
	\begin{lemma}
		Let $n \geq 4$ and $G_n$ a caterpillar graph with $|U| = 1$. Suppose that for all $v_i \in P_n$, the central path of $G_n$, $l(v_i) \in [1,2]$. Then $ \lambda_H(G_n) \geq \lceil\frac{\sum_{i=1}^{n}l(v_i)}{2} \rceil$.
	\end{lemma}
	\begin{proof}
		Suppose that $P_n$ is the central path of $G_n$ and for all $v_i \in V(P_n)$, let $\pi(v_i) = \sum_{1}^{n}l(v_i)$. From earlier result, $\lambda_H (G_n) \geq \lceil \frac{\pi(v_i)}{2} \rceil$.
	\end{proof}
	\begin{remark} \normalfont
		It should be observed that the position of $G_k \in G_n$ is particularly important in determining the Hamiltonian complete number of $G_n$.  
		\\
		\\
		Next we investigate the $\lambda_H-$number of an irregular caterpillar graph $G_n$, such that for every vertex $v_i$ on the central path $P_n$ of $G_n$, $l(v_i) \geq 3$. We consider variants of this class of caterpillar graphs.  
	\end{remark}
	\begin{theorem}
		Let $G_n$ be a caterpillar graph with central $P_n$ and suppose that for each $v_i \in V(P_n)$, $l(v_i) \geq  3$. Then, $\lambda_H (G) = \sum_{i=1}^{n}l(v_i)-n$.
	\end{theorem}
	\begin{proof}
		Let $v_i \in V(P_n)$ and let $v_iu_{i(1)}, v_iu_{i(2)}, \cdots, v_iu_{i(t)}$ be the pendants adjacent to $v_i$. Clearly, $l(v_i)=t$. It can be seen that the vertices $v_i, u_{i(1)}, u_{i(2)}, \cdots, u_{i(t)}$ induce a star $S(i)_{t+1}$. From the earlier result, $\lambda_H(S(i)_{t+1})=t-2=l(v_i)-2$. Thus, $\delta_H(S(i)_{t+i})=t(v_i)-2$. Thus, for all other stars induced by all the remaining vertices on $P_n$ and there adjacent leaves, the sum total of their Hamiltonian path complete number will be $\delta_H(\cup_{k=1}^nS(k))= \sum_{k=1}^{n}(l(v_k)-2)=\sum_{k-1}^{n}l(v_k)-2n$. Furthermore, for $G_n$ to contain a spanning path, every pair of induced neighboring vertices $v_k, v_{k+1}$ and their adjacent leaves will be connected by an external edge $e_k \notin E(G_n)$. Therefore, $\delta_H(G_n) \geq \sum_{k=1}^{n}l(v_k)-n$.
		
		Conversely, suppose that $E^H(G_n)$ is the Hamiltonian edge set of $G_n$, populated by external edges necessary to make $G_n$ Hamiltonian. Suppose, without loss of generality, that $v_1$ is the initial vertex on $P_n$ and $u_1,u_2,\cdots u_t \in V(G_n)$ are leaves that are adjacent to $v_1$, for all $i \in [1,t]$. There exists a path from $u_1 \rightarrow v_1 \rightarrow u_2 \rightarrow u_3 \rightarrow \cdots \rightarrow u_t$. Clearly, edges $u_2u_3, u_3u_4, \cdots , u_{t-1}u_t \in E^H(G_n)$. We repeat this for the other induced stars on $G_n$. Now, let $u_i(1)$ and $u_i(k)$ be two vertices on $S(i)_{t-1}$, which, without loss of generality, are considered as the first and last leaves respectively adjacent to vertex $v_i \in P_n$, for all $i \in [1,n]$. Then edges $u_1(k)u_2(1), u_2(k)u_3(1), \cdots, u_{n-1}(k)u_n(1), u_n(k)u_1(1)$ are edges in $E^H(G_n)$, which, in all, are $n$ external edges, connecting them to $G_n$ along with the other established $E^H(G_n)$ members, there is a spanning cycle in the new graph, and $|E^H(G_n)|= \sum_{i=1}^{n} l(v_i)-n$. 
	\end{proof}
	
	\begin{definition} \normalfont
		Let $G_n$ be a caterpillar graph on a central path $P_n$ and suppose that $\bar{P}$ is the set of subgraphs of $P_n$, such that members of $\bar{P}$ are subpaths of $P_n$ for which $\bar{P_t} \in \bar{P}$, $3 \leq t \leq n$, $l(v_1), l(v_t) \geq 2$ and $l(v_i) \in [0,1]$ for all intermediate $v_1$, $i\ \in [2,t-1] \in \bar{P_t}$. That is, $\bar{P}$ consists of subpaths of $P_n$ such that at least one of the end vertices is adjacent to at least two leaves while the rest of the intermediate vertices are adjacent to none or at most one leaf. We call $\bar{P}$ the set of $(0,1)-$leaf segment paths on $P_n$ and the $n(\bar{P})$ is the cardinal number of $\bar{P}$.
	\end{definition}
	\begin{definition} \normalfont
		If the intermediate vertices of a $(0,1)-$segment path is not adjacent to any leaf on $G_n$, then we call it a $0-leaf$ segment path of $P_n$. The cardinal number of the set of $0-$leaf segment path of $P_n$ is denoted as $P(0)$.
	\end{definition} 
	
	\begin{lemma}
		Suppose that $\bar{P_k}$ is a $0-$leaf segment on $P_n$ of $G_n$ with $l(v_1)=l(v_2)= \cdots l(v_{k-1})=0$ and $l(v_k) \geq 2$. Then $\delta_H(\bar{P_k})=l(v_k)-1$
	\end{lemma}   
	
	\begin{proof}
		Suppose that $u_1,u_2, \cdots, u_s$ are the leaves such that $v_ku_1, v_ku_2, \cdots , v_ku_s$ are pendants, then it is easy to see that $\delta_H (\bar{P_k})=l(v_k)-1$ either by introducing external edges $v_{k-1}u_1$ WLOG from $E^H(G_n)$ and $\delta_H(G(v_k))=l(v_k)-2$, where $G(v_k)$ is a subgraph induced by $v_k$ and it adjacent leaves, or by considering $\lambda_H(G(v_k))=l(v_k)-1$. 
	\end{proof}
	
	\begin{remark} \normalfont
		We adopt the external edge option such that $v_{k-1}u_1 \in E^H(G_n)$ in our subsequent proofs as the link between $0-$leaf segment and a claw. 
	\end{remark}
	\begin{theorem}
		Let $G_n$ be a caterpillar graph on a central path $P_n$. Suppose that for all $v_i \in V(P_n)$, either $l(v_i) \geq 2$ or $l(v_i)=0$. Suppose further that $G_n$ contains $P(0)$ number of $0-$leaf segment paths on $P_n$. Then, $\lambda_H(G_n)=P(0)+  \sum_{i=1; l(v_i) \geq 2}^{n}(l(v_i)-1)$.
	\end{theorem}
	\begin{proof}
		Let us consider $v_i \in V(P_n)$ such that $l(v_i) \geq 2$. Suppose that $v_i$ and the adjacent leaves induce a star $S_t,$ $t \geq 2$.The $\delta_H(S_t) = t-2$. Now, for $G_n$ to be Hamiltonian, there exists an edge connecting $S_t$ to a $0-$leaf segment path $P_w$, which is adjacent to $S_t$. Thus the edge, either $v_iv_{i+1}$, with $v_{i+1} \in V(P_n)$ or  $u_tv_{i+1}$, where $u_t$ is a leave on $S(t)$. Note that if we choose $v_iv_{i+1}$, then $G_n$ requires $t-1$ Hamiltonian complete edges on $S_t$. If $u_tv_{i+1}$ is selected, then it is an edge in $E^H(G_n)$. Then it can be seen that $\delta_H(S_t)=t-2$ and hence $t-1$ member of $E^H(G_n)$ is required as well. From the last lemma and remark, every $0-$leave segment on $P_n$ contributes an edge to $E^H(G_n)$ and thus, $|E^H(G_n)| = \lambda_H(G_n)=P(0)+ \sum_{i=1; l(v_i) \geq 2}^{n}(l(v_i-1))$.  
	\end{proof}
	Next, we define the concept of a deserted pendant
	
	\begin{definition} \normalfont
		Let $P_k, k\geq 3$ be a subgraph of a central path $P_n$ of a caterpillar graph $G_n$ and let $v_t \in V(P_k)$ such that $l(v_t)=1$. If $l(v_{t-1})=l(v_{t+1})=0$, then we say that pendant $v_tu_t$ is a deserted pendant.
	\end{definition}
	\begin{lemma}
		Suppose that for $k \geq 3$, $P_k$ is a subgraph of a central path $P_n$ of a caterpillar graph $G_n$, and $P_k$ contains a deserted pendant. Then, $\delta_H(P_k) \geq 1$. 
	\end{lemma}
	\begin{proof}
		Suppose that $v_t$ is a vertex on $P_k$ such that $v_tu_i$ is a deserted pendant. Then, $l(v_{k-1})=l(v_{k+1})=0$. Thus, for $P_k$ to attain an Hamiltonian path, either an external edge $v_{k-1}u_i$  or  $u_iv_{k-1}$  is introduced or an external edge.  
	\end{proof}
	\begin{cor}
		For each deserted pendant on $l_k$ on a caterpillar graph, $G_n$ and external leave is required for $G_n$ to contain a spanning path.   
	\end{cor}
	\begin{proof}
		The claim is obvious from the last Lemma.
	\end{proof}
	\begin{remark} \normalfont
		Now, we consider a special $(0,1)-$ pendant segment $\bar{P}$ in such a way that all the leaf $v_ku_i$ on $\bar{P}$ is a deserted pendant. We call $\bar{P}$ a $(0,1-)$deserted pendant segment.
	\end{remark}
	\begin{theorem}
		Suppose that $G_n$ is a caterpillar graph on a central path $P_n$, and suppose that for all $(0,1)-$pendant segment $\bar{P_i}$ on $G_n$, $\bar{P_n}$ is a $(0,1)-$deserted pendant segment. Then, $\lambda_H(G_n)= \sum_{_i=1; l(v_i) \geq 2}^{\gamma}(l(v_i)-1)+ \gamma + \tau$, where $\tau $ is the number of isolated pendants on $G_n$ and $\gamma$ is the number of $(0,1)-$isolated pendants segments on $P_n$.   
	\end{theorem}
	\begin{proof}
		Let $\bar{P_i}$ be some $(0,1)-$pendant segment such that every pendant attached to $\bar{P_n}$ is a isolated. Then, there exists vertex $v_k$ on $\bar{P}$ such that $v_k$ is adjacent to some $v_{k-1} \in V(P_n)$, where $l(v_k)=0$ and $l(v_{k-1}) \geq 2$. Suppose that $v_{k-1}$ and the leaves adjacent to it induce a $S_p$ subgraph of $G_n$ with $p$ pendants. Now, from earlier result, $\lambda_H(S_p) = l(v_{k-1})-1$. Furthermore, let $v_{k+t} \in \bar{P_i}$ be such that $l(v_{k+t})=1$ and $v_{k+t}$ is the first vertex from $v_{k-1}$ on $\bar{P_i}$ adjacent to a leaf to form an isolated pendant $v_{k+t}u_k=l_t$, and $d(v_k,v_{k+t})=t$. With $l_t$ being an isolated pendant, there exists some vertex $v_{k+t+1}$ adjacent to $v_{k+t}$, such that $l(v_{k+t+1})=0$ and thus, there exists some edge $v_{k+t+1}u_k \in E^H(G_n)$ necessary  for $G_n$ to be Hamiltonian. Now, there exists $v_j \in \bar{P_i}$ such that $j > k+t+1$, such that $d(v_J, v_{K+t+1}) \geq 1$ and $l(v_j) \geq 1$. Suppose that $l(v_j)=1$, then there exists a pendant $v_ju_j$, which is an isolated pendant. Then there exits $ u_jv_{j+1} \in E^H(G_n)$ connecting $v_ju_j$ to vertex $v_{j+1}$. Suppose this process continues through $r-$isolated pendants on $P_i$. There still exits a vertex $v_s \in \bar{P_i}$ with $l(v_s) = 0$, such that $l(v_{s+1}) \geq 2$ and $v_{s+1} \in P_n$. Thus, there exits some edge $v_su_0 \in E^H(G_n)$, such that $v_{s+1}u_0$, such that $v_{s+1}u_0$ is a pendant on the subgraph $S_q$ of $G_n$, induced by $v_{s+1}$ and the leafs incident to $v_{s+1}$. Thus if $\bar{P_i}$ contains $t$ isolated vertices, then $\bar{P_i}$ necessarily requires $t+1$ edges from $E^H(G_n)$ for $G_n$ to be Hamiltonian. Now suppose that $P_n$ contains $\gamma$ number of $(0,1)-$isolated pendant segments, with $\delta_1, \delta_2, \cdots, \delta_{\gamma}$ isolated pendants and $\delta_k \geq 0, k \in [1,\gamma]$. Therefore, the combination of all $(0,1)-$isolated pendant segments in $G_n$ necessarily requires $\gamma + \sum_{i=1}^\gamma \delta_1$ edges from $E^H(G_n)$ to make $G_n$ Hamiltonian. Next, let $j$ be some positive integer such that $j \in [1, \gamma +1]$. Now, there exist the vertices $v_i, v_2, \cdots, v_j$ such that for all $i \in [1,j]$, $l(v_i) \geq 2$. Then from earlier, $\lambda_H(G_n) \leq \sum_{i = 1}^{j}(l(v_i)-1)+\sum_{i = 1}^{\gamma} \delta_i + \gamma$. 
		Conversely, suppose that $v_1, v_1 \in V(P_n)$ such that $l(v_1) \geq 2$ and $l(v_2)=0$. From an earlier result, $\delta_H(s_k)=l(v_1)-2$, where $S_k$ is a subgraph of $G_n$, induced by $v_1$ and the leaves adjacent to $v_1$. Now, suppose that $v_2 \bar{P_k}$ a $(0,1)-$pendant segment adjacent to $S_k$. Then, from earlier remark and WLOG, there exists $u_k \in V(S_k)$ such that $u_kv_2 \in E^H(G_n)$. Then, for $S_p \cup \bar{P_k} \subseteq G_n$, and thus, $\lambda_H(S \cup \bar{P}) \geq l(v_i)-1$. It is easy to trace the claim from this point. 
	\end{proof}
	
	\section{Conclusion}
	
	In this paper, we have investigated the Hamiltonian complete number \(\lambda_H(G)\) for various classes of caterpillar graphs. We derived explicit formulas for \(\lambda_H(G)\) in the case of regular caterpillar graphs, where each vertex on the central path is adjacent to a fixed number of leaves. Specifically, we showed that for a regular caterpillar graph \(G_{n(k)}\), \(\lambda_H(G_{n(k)}) = n(k-1)\). We also explored irregular caterpillar graphs, where the number of leaves varies, and provided bounds for \(\lambda_H(G)\) in these cases. Our results extend the understanding of Hamiltonian properties in tree-like structures and have potential applications in network design, route planning, and optimization. Future work could explore the Hamiltonian complete number for other classes of graphs or investigate the computational complexity of determining \(\lambda_H(G)\) for more complex graph structures.


\begin{thebibliography}{99}	
		\bibitem[1]{A1} Adefokun, T.C. {\it{Hamiltonian Complete Number of Certain Graph Classes}}, Crawford Journal of Natural and Applied Science (Accepted for Publication).
		
		\bibitem[2]{AA1} Ajayi, D.O. and Adefokun, T.C. {\it{Condition for Hamiltoniancity of Certain Special Graphs}}, Acta Universitatis Apulensis 39, 301-304, (2014).
		
		\bibitem[3]{BC1} Bondy, A. and Chavtal, V. {\it{A Method in Graph Theory}}, Discrete Math, 15, 111-135 (1976).
		
		\bibitem[4]{D1} Dirac, G.A. {\it{Some Theorems on Abstract Graphs}}, Proceedings of the London Mathematical Society, Series 3 (2) 69-81 (1952).
		
		\bibitem[5]{EN11} Elligham, M and Norton, J. {\it{Non-Hamilton 3-Connected Cubic Bipartite Graph}},Jour. of Combin. Theory, B (34) 350-353 (1983).
		
		
		\bibitem[6]{H1} Hassler, W. {\it{A Theorem on Graphs}}, Ann. of Math. 32, 378-390 (1931).
		
		
		\bibitem[7]{O1} Ore, O {\it{A Note on Hamilton Circuits}}, the American Mathematical Monthly, 6 7(1) 55 (1960).
		
		\bibitem[8]{R1} Redel, L., {\it{Ein Kombinastorisher Satz}}, Acta Litt. Sci. Szeged, 39-43 (1943).
		
		\bibitem[9]{SB} Sleegers, J., van den Berg, D. The Hardest Hamiltonian Cycle Problem Instances: The Plateau of Yes and the Cliff of No. SN COMPUT. SCI. 3, 372 (2022). https://doi.org/10.1007/s42979-022-01256-0	
		
		\bibitem[10]{T1} Tutte W.T. {\it{A Theorem on Planar Graph}}, Trans. Amer. Math. Soc. 82, 99-119, 1956.
		
		
		
		
	
		
		
		
	\end{thebibliography}
\end{document}